\documentclass{baustms}

\citesort

\theoremstyle{cupthm}
\newtheorem{thm}{Theorem}[section]
\newtheorem{prop}[thm]{Proposition}
\newtheorem{cor}[thm]{Corollary}
\newtheorem{lemma}[thm]{Lemma}
\theoremstyle{cupdefn}

\theoremstyle{cuprem}

\numberwithin{equation}{section}

\begin{document}
\runningtitle{$I$-primary}  
\title{$I$-primary submodules}
\author[1]{Ismael Akray}
\address[1]{Department of Mathematics, Soran University, Erbil city, Kurdistan region, Iraq\email{ismael.akray@soran.edu.iq}}
\author[2]{Halgurd S. Hussein}
\address[2]{Department of Mathematics, Soran University, Erbil city, Kurdistan region, Iraq\email{rasty.rosty1@gmail.com}}

\authorheadline{I. Akray and H. S. Hussein}



\begin{abstract} 
   \hspace{.3cm} In this paper, we give a generalization for weakly primary submodules called $I$-primary submodule and we study some properties of it. We give some characterizations of $I$-primary submodules. Also we establish the situation of $I$-primary submodules in module localizations and decomposable modules.
\end{abstract}


\maketitle

\section{Introduction}

\hspace{.3cm} Throughout this paper $R$ will be a commutative ring with nonzero identity and $I$ a fixed ideal of $R$ and $M$ a unitary left $R$-module. The concept of weakly prime ideals was introduced by Anderson and Smith (2003), where a proper ideal $P$ is called weakly prime if, for $a,b \in R$ with $0 \neq ab \in P$, either $a \in P$ or $b \in P$, \cite{6-6}. The radical of an ideal $I$ of $R$ is defined to be the set of all $a \in R$ for which $a^n \in I$ for some positive integer $n$. Primary ideals have an important role in commutative ring theory. A proper ideal $Q$ of $R$ is said to be primary provided that for $a,b \in R, ab \in Q$ implies that either $a \in Q$ or $b \in \sqrt{Q}$. We can generalize the concept of primary ideals by restricting the set where $ab$ lies. A proper ideal $Q$ of $R$ is weakly primary if for $a,b \in R$ with $0 \neq ab \in Q$, either $a \in Q$ or $b \in \sqrt{Q}$. Weakly primary ideals were first introduced and studied by Ebrahimi Atani and Farzalipour in 2005, \cite{5-5}.

An $R$-module $M$ is called a multiplication module if every submodule $N$ of $M$ has the form $IM$ for some ideal $I$ of $R$, see \cite{34}. Note that, since $I \subseteq (N :_R M)$ then $N = IM \subseteq (N :_R M)M \subseteq N$. So that $N = (N :_R M)M$. Let $N$ and $K$ be submodules of a multiplication $R$-module $M$ with $N = I_1M$ and $K = I_2M$ for some ideals $I_1$ and $I_2$ of $R$. The product of $N$ and $K$ denoted by $NK$ is defined by $NK = I_1I_2M$. Then by \cite[Theorem 3.4]{123}, the product of $N$ and $K$ is independent of presentations of $N$ and $K$. An $R$-module $M$ is called faithful if it has zero annihilator. Let $N$ be a proper submodule of a nonzero $R$-module $M$. Then the $M$-radical of $N$, denoted by $M-rad(N)$, is defined to be the intersection of all prime submodules of $M$ containing $N$. If $M$ has no prime submodule containing $N$, then we say $M-rad(N) = M$. It is shown in \cite[Theorem 2.12]{1234} that if $N$ is a proper submodule of a multiplication $R$-module $M$, then $M-rad(N) = \sqrt{(N :_R M)}M$. 

The class of prime submodules of modules was introduced and studied in 1992 as a generalization of the class of prime ideals of rings. Then, many generalizations of prime submodules were studied such as primary, classical prime, weakly prime and classical primary submodules, see \cite{34,124, 134, 164} and \cite{114}. The authors in \cite{11-1} and \cite{22-2} introduced the notions $I$-prime ideals and $I$-prime submodules. A proper ideal of $R$ is called $I$-prime ideal of $R$ if $rs \in P-IP $ for all $r, s \in R$ implies that either $r \in P$ or $s \in P$. A proper submodule $P$ of $M$ is called $I$-prime submodule of $M$ if $rm \in P-IP $ for all $r \in R$ and $m \in M$ implies that either $m \in P$ or $r \in (P:M)$. In this article, we define and study $I$-primary submodules which are generalizations of weakly primary (and weakly prime) submodules. We generalize some basic properties of weakly prime and weakly primary to $I$-primary submodules. We give some characterizations of $I$-primary submodules.

\section{Main results}

Let $I$ be an ideal of $R$ and $M$ an $R$-module. A proper submodule $P$ of $M$ is called an $I$ -primary submodule of $M$ if $rm \in P-IP $ for all $r \in R$ and $m \in M$ implies that either $m \in P$ or $r \in \sqrt{(P:M)}$. In view of the above definition, in all what follows, $I$ is an ideal of $R$.
Now we begin with the following result.

		\begin{thm}
			Let $M$ be an $R$-module. Let $P$ be an $I$-primary submodule of $M$. If $(P:M)P \nsubseteq IP$, then $P$ is a primary submodule of $M$. 
		\end{thm}

	\begin{proof}
		Suppose that $r \in R$ and $m \in M$ such that $rm \in P$. If $rm \notin IP$, since $P$ is $I$-primary submodule of $M$, then $r^{n} \in (P:M)$ for some positive integer $n$ or $m \in P$. Now suppose that $rm \in IP$. Let $rP \subseteq IP$. Because if $rP \nsubseteq IP$, then there exists $p \in P$ such that $rp \notin IP$, then $r(m+p) \in P-IP$. Therefore $r^{n} \in (P:M)$ for some positive integer $n$ or $m+p \in P$. Hence $r^n \in (P:M)$ for some positive integer $n$ or $m \in p$. Now suppose that $(P:M)m \subseteq IP$. Because if $(P:M)m \nsubseteq IP$, then there exists $t \in (P:M)$ such that $tm \notin IP$ and so $(r+t)m \in P-IP$. Then we have $r+t \in \sqrt{(P:M)}$ or $m \in P$ . Hence $r \in \sqrt{(P:M)}$ or $m \in P$. Since $(P:M)P \nsubseteq IP$, so there exist $b \in (P:M)$ and $p \in P$ such that $bp \notin IP$. Then $(r+b)(m+p) \in P-IP$. Therefore $r+b \in \sqrt {(P:M)}$ or $m+p \in P$. Hence $ r \in \sqrt 
		{(P:M)}$ or $m \in P$. Thus $P$ is primary submodule in $M$.
	\end{proof}

		\begin{cor}
			Let $P$ be a $0$-primary of $M$ such that $(P:M)P  \neq 0$. Then $P$ is a primary submodule of $M$.
		\end{cor}

		\begin{cor}
			Let $P$ be an $I$-primary submodule of $M$ such that $IP \subseteq (P:M)^{2}P$ and that $P$ is not primary submodule. Then $IP =(P:M)^{2}P$.
		\end{cor}
	
	In what follows we give some charactrizations for $I$-primary submodules.
	
	\begin{thm} \label{1}
		Let $P$ be a proper submodule of $M$, then the following are equivalent:
		\\(1) $P$ is an $I$-primary submodule of $M$.
		\\(2)For $r  \in R- \sqrt{(P:M)},\ (P:r)=P \cup (IP:r)$.
		\\(3) For $r  \in R- \sqrt{(P:M)},\ (P:r)=P$ or $(P:r)=(IP:r)$.
	\end{thm}
	\begin{proof}
		(1)$\Rightarrow$ (2) Suppose that $P$ is an $I$-primary submodule of $M$ such that $r \notin \sqrt{(P:M)}$. Let $a \in (P:r)$. So $ra \in P$. If $ra \notin IP$, then $a \in P$. Because $P$ is an $I$-primary submodule in $M$. If $ra \in IP$, then $a \in (IP:r)$. So $(P:r) \subseteq P \cup (IP:r)$. Now since $IP \subseteq P$, the other inclusion is hold.
		\\
		(2)$\Rightarrow$ (3) It is clear because $(P:r)$ is a submodule of $M$.
		\\
		(3)$\Rightarrow$ (1) Let $r \in R$ and $x \in M$ such that $rx \in P-IP$. If $r \notin \sqrt{(P:M)}$, then by assumption, either $(P:r)=P$ or $(P:r)=(IP:r)$. Since $rx \notin IP$ then $x \notin (IP:r)$, and since $rx \in P$, then $x \in (P:r)$. So $(P:r)=P$. Therefore $x \in P$. Thus $P$ is an $I$-primary submodule of $M$. 
	\end{proof}
	
	The quotient and localization of primary submodules are again primary submodules. But in case of $I$-primary submodules, we give a condition under which the localization becomes true as we see in the following theorem.

	\begin{thm}
		Let $M$ be an $R$-module. Let $P$ be an $I$-primary submodule of $M$. Then,
		\\(i) If $N \subseteq P$ is a submodule of $M$, then $ \frac{P}{N}$ is an $I$-primary submodule of $\frac{M}{N}$.
		\\(ii) Suppose that $S$ is a multiplicatively closed subset of $R$ such that $S^{-1}P \neq S^{-1}M$ and $S^{-1}(IP) \subseteq (S^{-1}I)(S^{-1}P)$. Then $S^{-1}P$ is an $(S^{-1}I)$-primary submodule of $S^{-1}M$.
	\end{thm}
	
	\begin{proof}
			(i) Let $r\in R$ and $m \in M$ such that $r(m+N)=rm+N \in \frac{P}{N}-I \frac{P}{N}$. Then $rm+N \in \frac{P-IP}{N}$. So $rm \in P-IP$. As $P$ is $I$-primary submodule, then either $r^{n} \in (P:M)$ for some positive integer $n$ or $m \in P$. Therefore $r^{n} \in (\frac{P}{N}:\frac{M}{N})$ for some positive integer $n$ or $m+N \in \frac{P}{N}$. Hence $\frac{P}{N}$ is $I$-primary submodule of $\frac{M}{N}$.
		
			(ii) For all $\frac{r}{s}\in S^{-1}R$ and $\frac{m}{t} \in S^{-1}M$, let $\frac{r}{s}.\frac{m}{t}=\frac{rm}{st} \in S^{-1}P-(S^{-1}I)(S^{-1}P) \subseteq S^{-1}P-S^{-1}(IP)=S^{-1}(P-IP)$. Then $\frac{rm}{st}= \frac{x}{u}$ for $x \in P-IP$ and $u \in S$. So for some $q \in S, qurm=qstx \in P-IP$. As $P$ is $I$-primary submodule , $(qru)^{n} \in (P:M)$ for some positive  integer $n $ or $m \in P$. So $\frac{(qur)^{n}}{(qus)^{n}}= (\frac {r}{s})^{n} \in S^{-1}(P:M)=(S^{-1}P:_{S^{-1}R}S^{-1}M)$ or $ \frac{m}{t} \in S^{-1}P$. Hence $S^{-1}P$ is $(S^{-1}I)$-primary submodule of $S^{-1}M$.
		\end{proof}

	\begin{prop}
		Let $M$ be an $R$-module and $P$ be a submodule of $M$.
		\\(i) If $I_{1} \subseteq I_{2}$. Then $P$ is $I_{1}$-primary implies $P$ is $I_{2}$-primary.
		\\(ii) If $P$ is $I$-prime then $P$ is $I$-primary.
	\end{prop}
	
	\begin{proof}
		(1) Suppose that $P$ is $I_{1}$-primary.Let $r \in R, m \in M$ with $rm \in P-I_{2}P$. Since $I_{1} \subseteq I_{2},\ P-I_{2}P \subseteq P-I_{1}P$. Then $rm \in P-I_{1}P$. But $P$ is $I_{1}$-primary. So $r \in \sqrt{(P:M)}$ or $m \in P$. Thus $P$ is $I_{2}$-primary.
		\\ (2) The proof is trivial.
	\end{proof}

	An ideal $I$ is called radical if $I= \sqrt{I}$. In the following we give a condition under which $I$-prime and $I$-primary be equivalent.
	\begin{prop}
		Let $P$ be a proper submodule of an $R$-module $M$ such that $(P:M)$ is radical ideal. Then $P$ is $I$-primary if and only if $P$ is $I$-prime.
	\end{prop}
	\begin{proof}
		The proof comes from the definitions. 
	\end{proof}

	\begin{thm}
		Let $M_{1}$ and $M_{2}$ be modules over $R_{1}$ and $R_{2}$ respectively with $M=M_{1} \times M_{2}$. Suppose that $P_{1}$ is an $I_{1}$-primary submodule of $M_{1}$ such that $IP_{1} \times M_{2} \subseteq I(P_{1} \times M_{2})$. Then $P_{1} \times M_{2}$ is an $I$-primary submodule of $M_{1} \times M_{2}$. 
	\end{thm}

	\begin{proof}
	Let $(r_{1},r_{2}) \in R$, and $(x_{1},x_{2}) \in M$ with $(r_{1},r_{2})(x_{1},x_{2})=(r_{1}x_{1},r_{2}x_{2}) \in P_{1} \times M_{2}-I(P_{1} \times M_{2})$, and $P_{1} \times M_{2}-I(P_{1} \times M_{2}) \subseteq P_{1} \times M_{2}-IP_{1} \times M_{2}=(P_{1}-IP_{1}) \times M_{2}$. We have $r_{1}x_{1} \in P_{1}-IP_{1}$ but $P_{1}$ is $I$-primary submodule of $M_{1}$. Then $r_{1}^{n} \in (P_{1}:_{R_{1}}M_{1})$ for some positive integer  $n$ or $x_{1} \in P_{1}$. So $(r_{1},r_{2})^{n}=(r_{1}^{n},r_{2}^{n}) \in (P_{1}:_{R_{1}}M_{1}) \times R_{2}=(P_{1} \times M_{2}:_{R_{1}\times R_{2}}M_{1} \times M_{2})$ for some positive integer  $n$ or $(x_{1},x_{2}) \in P_{1} \times M_{2}$. Hence $P_{1} \times M_{2}$ is an $I$-primary submodule of $M_{1} \times M_{2}$.
	\end{proof}


We illustrate the formation of the $I$-primary submodules in decomposition modules. In other manner the formation of $I$-primary submodules in decomposable module have one of the five types stated in the following theorem.	

	\begin{thm} \label{a1}
		Let $R=R_{1} \times R_{2}$. Let $M_{1}$ and $M_{2}$ be $R_{1}$ and $R_{2}$-modules respectively with $M=M_{1} \times M_{2}$. Let $I_{1}$ and $I_{2}$ be ideals of $R_{1}$ and $R_{2}$ respectively with $I=I_{1} \times I_{2}$. Then all the following types are $I$-primary submodule of $M_{1} \times M_{2}$.
		\begin{enumerate}
		\item $P_{1} \times P_{2}$ where $P_{i}$ is a proper submodule of $M_{i}$ with $I_{i}P_{i}=P_{i}$ for $i=1, 2$.
		\item $P_{1} \times M_{2}$ where $P_{1}$ is a primary submodule of $M_{1}$.
		\item $P_{1} \times M_{2}$ where $P_{1}$ is an $I_{1}$-primary submodule of $M_{1}$ and $I_{2}M_{2}=M_{2}$.
		\item $M_{1} \times P_{2}$ where $P_{2}$ is a primary submodule of $M_{2}$.
		\item $M_{1} \times P_{2}$ where $P_{2}$ is an $I_{2}$-primary submodule of $M_{2}$ and $I_{1}M_{1}=M_{1}$.
		\end{enumerate}
		\end{thm}
		
	\begin{proof}
	\begin{enumerate}
	\item Since $I_{1}P_{1}=P_{1}$ and $I_{2}P_{2}=P_{2}$. Then $I_{1}P_{1} \times I_{2}P_{2}= I{_1} \times I_{2}(P_{1} \times P_{2})= I(P_{1} \times P_{2})=P_{1} \times P_{2}$. So $P_{1} \times P_{2}= I(P_{1} \times P_{2})= \phi $. Thus there is nothing to prove. 
	\item Let $P_{1}$ be a primary submodule of $M_{1}$. Then $P_{1} \times M_{2}$ is a primary submodule of $M_{1} \times M_{2}$. Thus $P_{1} \times M_{2}$ is $I$-primary submodule of $M$. 
	\item Suppose that $P_{1}$ is an $I_{1}$-primary submodule of $M_{1}$ and $I_{2}M_{2}=M_{2}$. Let $(r_{1},r_{2}) \in R$ and $(m_{1},m_{2}) \in M$ such that $(r_{1},r_{2})(m_{1},m_{2})=(r_{1}m_{1},r_{2}m_{2}) \in P_{1} \times M_{2}-I(P_{1} \times M_{2})= P_{1} \times M_{2}-(I_{1} \times I_{2})(P_{1} \times M_{2})= P_{1} \times M_{2}-(I_{1}P_{1} \times I_{2}M_{2})= P_{1} \times M_{2}-(I_{1}P_{1} \times M_{2})=(P_{1}-I_{1}P_{1}) \times M_{2}$. Then $r_{1}m_{1} \in P_{1}-I_{1}P_{1}$ and $P_{1}$ is $I_{1}$-primary submodule of $M_{1}$, so $r_{1}^{n} \in (P_{1}:_{R_{1}}M_{1})$ for some positive integer $n \in N$ or $m_{1} \in P_{1}$. Therefore $(r_{1},r_{2})^{n}=(r_{1}^{n},r_{2}^{n}) \in (P_{1}:_{R_{1}}M_{1}) \times R_{2}=(P_{1} \times M_{2}:_{R_{1} \times R_{2}} M_{1} \times M_{2})$ for some positive integer $n \in N$ or $(m_{1},m_{2}) \in P_{1} \times M_{2}$. So $P_{1} \times M_{2}$ is an $I$-primary submodule of $M_{1} \times M_{2}$. 
	\item The proof is similar to part (2).
	\item The proof is similar to part (3).
	\end{enumerate}
		
	\end{proof}

	\begin{prop}
		Let $M$ be an $R$-module and let $P$ be a proper submodule of $M$. Then $P$ is $I$-primary in $M$ if and only if $\frac{P}{IP}$ is $0$-primary in $\frac{M}{IP}$. 
	\end{prop}
	\begin{proof}
		Suppose that $P$ is $I$-primary in $M$. Let $r \in R, m \in M$ such that $0 \neq rm+IP=r(m+IP) \in \frac{P}{IP}$ in $\frac{M}{IP}$. Then $rm \in P-IP$. But by assuming $P$ is $I$-primary, so $m \in P$ or $r\in \sqrt{(P:M)}$. Then $r \in \sqrt{P:M}=\sqrt{(\frac{P}{IP}:\frac{M}{IP})}$ or $m+IP \in \frac{P}{IP}$. Hence $\frac{P}{IP}$ is $0$-primary submodule in $\frac{M}{IP}$. Conversely suppose that $\frac{P}{IP}$ is $0$-primary submodule in $\frac{M}{IP}$. Let $r\in R,m \in M$ such that $rm \in P-IP$. So $0 \neq r(m+IP)=rm+IP \in \frac{P}{IP}$. But $\frac{P}{IP}$ is $\{0\}$-primary in $\frac{M}{IP}$. Thus $m+IP \in \frac{P}{IP}$ or $r \in \sqrt{(\frac{P}{IP}:\frac{M}{IP})}$ and so $m \in P$ or $r \in \sqrt{(P:M)}$. Hence $P$ is $0$-primary.  
	\end{proof}
	
	Now we give two other charactrizations of $I$-primary submodules.
	\begin{thm}
		Let $M$ be an $R$-module and $P$ be a proper submodule of $M$. Then $P$ is $I$-primary in $M$ if and only if for any ideal $J$ of $R$ and submodule $N$ of $M$ such that $JN \subseteq P-IP$, we have $J \subseteq \sqrt{(P:M)}$ or $N \subseteq P$
	\end{thm}
	\begin{proof}
		Suppose that $P$ is $I$-primary submodule of $M$, and $JN \subseteq P-IP$ for some ideal $J$ of $R$ and submodule $N$ of $M$. If $J \nsubseteq \sqrt{(P:M)}$ and $N \nsubseteq P$, so there exists $r \in J-\sqrt{(P:M)}$ and $x\in N-P$ such that $rx \in P-IP$. By assuming $P$ is $I$-primary in $M$, either $x\in P$ or $r\in \sqrt{(P:M)}$ which is a contradiction. Hence $J \subseteq \sqrt{(P:M)}$ or $N\subseteq P$. Conversely suppose that $rm\in P-IP$ for $r\in R$ and $m \in M$. Then $(r)(m)=(rm) \subseteq P-IP$. So by assumption, either $(r) \subseteq \sqrt{(P:M)}$ or $(m) \subseteq P$. Therefore $r\in \sqrt{(P:M)}$ or $m \in P$. Thus $P$ is an $I$-primary submodule of $M$. 
	\end{proof}

	The following theorem gives the relation between $I$-primary ideals and $I$-primary submodules.
	\begin{thm} \label{1.1}
		Let $M$ be a finitely generated faithful multiplication $R$-module and $P$ a proper submodule of $M$ with $(IP:M)=I(P:M)$. Then $P$ is an $I$-primary submodule of $M$ if and only if  $(P:M)$ is an $I$-primary ideal of $R$.
	\end{thm}

	\begin{proof}
		Assume that $P$ is $I$-primary submodule in $M$. Let $r,s \in R$, such that $rs \in (P:M)-I(P:M)$. Then $rsM \subseteq P$. If $rsM \subseteq IP$. Then $rs \in (IP:M)=I(P:M)$ which is contradiction. Assume $rsM \nsubseteq IP$. Then $rsM \subseteq P-IP$. So $r \in \sqrt{(P:M)}$ or $sM \subseteq P$. Therefore $r \in \sqrt{(P:M)}$ or $s \in (P:M)$. Hence $(P:M)$ is $I$-primary ideal. Conversely assume that $(P:M)$ is $I$-primary ideal. Let $r \in R, m \in M$ such that $rm\in P-IP$. $r(Rm:M) = (rRm:M) \subseteq (P:M)$ and $r(Rm:M) \nsubseteq I(P:M)$ otherwise $rRm=r(Rm:M)M \subseteq I(P:M)M=IP$. That is $r(Rm:M) \subseteq (P:M)-I(P:M)$. Therefore $r \in \sqrt{(P:M)}$ or $(Rm:M)\subseteq (P:M)$. Hence $r\in \sqrt{(P:M)}$ or $Rm\subseteq P$. Thus  $r\in \sqrt{(P:M)}$ or $m \in P$. This means that $P$ is $I$-primary submodule. 
	\end{proof}

	\begin{thm}
			Let $M$ be a finitely generated faithful multiplication $R$-module and $P$ be a proper submodule of $M$ such that $I(P:M)=(IP:M)$. Then $P$ is $I$-primary in $M$ if and only if whenever $N$ and $K$ are submodules of $M$ such that $NK \subseteq P-IP$, we have either $N \subseteq P$ or $ K \subseteq \sqrt{P}$.   
		\end{thm}
		\begin{proof}
			Suppose that $P$ is an $I$-primary submodule in $M$. So $(N:P)$ is an $I$-primary ideal in $R$. As $M$ is a multiplication $R$-module, $N=(N:M)M$, $K=(K:M)M$ and so $N.K=(N:M)(K:M)M$. Suppose that $NK \subseteq P-IP$, but $N \nsubseteq P$ and $K \nsubseteq \sqrt{P}$. Thus $(N:M) \nsubseteq (P:M)$ and $(K:M) \nsubseteq \sqrt{(P:M)}$. As $(P:M)$ is $I$-primary in $R$. So either $(N:M)(K:M) \nsubseteq (P:M)$ or $(N:M)(K:M) \subseteq I(P:M)$. In the first case, we have $N.K=(N:M)(K:M)M \subseteq (P:M)M=P$ which is contradiction. Now if $(N:M)(K:M) \subseteq I(P:M)$, then $N.K=(N:M)(K:M)M \subseteq I(P:M)M=IP$, which is again a contradiction. Hence either $N \subseteq P$ or $K \subseteq P$. To prove the converse part, by Theorem \ref{1.1}, it is enough to prove that $(P:M)$ is $I$-primary ideal in $R$. Let $a,b \in R$ such that $ab \in (P:M)-I(P:M)$ with $a \notin (P:M)$ and $b \notin (P:M)$. Take $N=aM, K=bM$. Then $N.K=abM \subseteq (P:M)M=P$. If $NK=abM \subseteq IP$, then $ab \in (IP:M)=I(P:M)$, which is contradiction. Hence $NK \subseteq P-IP$. By hypothesis and since $M$ is a faithful multiplication module, we have either $N=aM \subseteq P$ or $K=bM \subseteq \sqrt{P} =\sqrt{(P:M)M}$ and so $a \in (P:M)$ or $b \in \sqrt{(P:M)}$. Therefore $(P:M)$ is an $I$-primary ideal of $R$.
		\end{proof}

	\begin{prop}
		Let $P$ be a submodule of an $R$-module $M$ and $N$ be any $R$-module. Then $P$ is an $I$-primary submodule of $M$ if and only if $P\oplus N$ is an $I$-primary submodule of $M \oplus N$. 
	\end{prop}
	\begin{proof}
		Assume that $P$ is an $I$-primary submodule in $M$. Let $r \in R, (m,n) \in M \oplus N$ such that $r(m,n) \in (P \oplus N)-I(P \oplus N)$. Then $rm \in P-IP$. But $P$ is $I$-primary in $M$, so $m \in P$ or $r \in \sqrt{(P:M)}$, that is $(m,n) \in P \oplus N$ or $r \in \sqrt{(P \oplus N:M \oplus N)}$. Therefore $P \oplus M$ is $I$-primary in $M \oplus N$. Conversely suppose that $P \oplus N$ is $I$-primary in $M \oplus N$. Let $r \in R, m \in M$ such that $rm \in P-IP$. Then $r(m,0) \in (P \oplus N)-I(P \oplus N)$. But by assuming $P \oplus N$ is $I$-primary, so $(m,0) \in P \oplus N$ or $r \in \sqrt{(P \oplus N:M \oplus N)}$. Therefore $m \in P$ or $r \in \sqrt{(P:M)}$. Hence $P$ is $I$-primary in $M$.
	\end{proof}


	Let $J$ be an ideal. We show that under a certain condition the $I$-primaryness of submodules $P$ and $(P:J)$ are equivalent. First we give the following lemma.
	\begin{lemma} \label{2}
		Let $P$ be a submodule of a faithful multiplication $R$-module $M$ and $J$ be finitely generated faithful multiplication ideal of $R$, then:\\
		(1) $P=(JP:J)$;\\
		(2) If $P \subseteq JM$, then $(KP:J)=K(P:J)$ for any ideal $K$ of $R$.
	\end{lemma}

	\begin{prop}
		Let $P$ be a submodule of a faithful multiplication $R$-module $M$ and $J$ be a finitely generated faithful multiplication ideal of $R$. Then $P$ is an $I$-primary submodule of $JM$ if and only if $(P:J)$ is $I$-primary in $M$.
	\end{prop}

	\begin{proof}
		Suppose $P$ is $I$-primary in $JM$. Take $r \in R, m \in M$ such that $rm \in (P:J)-I(P:J)$. Then $rJm \subseteq P-IP$. If $rJm \nsubseteq IP$, then by Lemma \ref{2} , $rm \in (IP:J)=I(P:J)$ which is a contradiction. As $P$ is an $I$-primary in $JM$, then $Jm \subseteq P$ or $r \in \sqrt{(P:JM)}$. So $m \in (P:J)$ or $r \in \sqrt{((P:J):M)}$, because $(P:JM)=((P:J):M)$. Hence $(P:J)$ is an $I$-primary submodule in $M$. Conversely, assume that $(P:J)$ is $I$-primary in $M$. Let $K$ be an ideal of $R$, and $N$ be a submodule of $JM$ such that $KN \subseteq P-IP$. Then $K(N:J) \subseteq (KN:J) \subseteq (P:J)$. Moreover, if $K(N:J) \subseteq I(P:J)$, then by Lemma \ref{2} $KN=K(JN:J)=K(N:J)J\subseteq I(P:J)J=IP$ which is a contradiction. Hence $K(N:J)\subseteq (P:J)- I(P:J)$. As $(P:J)$ is $I$-primary in $M$, either $K\subseteq \sqrt{((P:J):M)}=\sqrt{(P:JM)}$ or $(N:J)\subseteq (P:J)$, which implies that $N=(N:J)J\subseteq (P:J)P=P$. Hence $P$ is an $I$-primary submodule in $JM$.
	\end{proof}

		\hspace{.3cm} Let $M$ and $F$ be $R$-modules and $r \in R$. Then it is clear that for any submodule $P$ of $M$, $F \otimes (P:r) \subseteq (F \otimes P:r)$. In the following lemma we give a condition under which the equality holds.
		\begin{lemma} \label{a2}
			Let $r \in R$ and $P$ a submodule of $M$. Then for every flat $R$-module $F$, we have $F\otimes(P:r)=(F \otimes P:r)$. 
		\end{lemma}
	
		\begin{proof}
			Consider the exact sequence $0\longrightarrow (P:r)\longrightarrow M \stackrel{f_r}{\longrightarrow} \frac{M}{P} $ where $f_{r}(m)=rm+P$. As $F$ is flat , the exactness of the sequence $0\longrightarrow P\longrightarrow M\longrightarrow \frac{M}{P} \longrightarrow 0$ implies the exactness of the sequence $0\longrightarrow F \otimes P \longrightarrow F \otimes M \longrightarrow F \otimes \frac{M}{P} \longrightarrow 0$ which gives $ F \otimes \frac{M}{P} \cong \frac{F \otimes M}{F \otimes P}$. So the exactness of the sequence $0\longrightarrow (P:r)\longrightarrow M \longrightarrow \frac{M}{P}$ imply the exactness of the sequence $0\longrightarrow F \otimes (P:r) \longrightarrow F \otimes M       \stackrel{1 \otimes \acute {f_{r}}}{\longrightarrow} \frac{F \otimes M}{F \otimes P}$ where $(1 \otimes \acute {f_{r}}) (n \otimes m)=r.(n \otimes m)+ F \otimes P$ for $n \in F$. Therefore $F \otimes (P:r)= ker(1 \otimes \acute{f_{r}})=(F \otimes P:_{F \otimes M}r)$.
		\end{proof}

		\begin{thm} \label{110}
			Let $P$ be an $I$-primary submodule of an $R$-module $M$ and $F$ be a flat $R$-module with $F \otimes P \neq F \otimes M$. Then $F \otimes P$ is an $I$-primary submodule of $F \otimes M$. 
		\end{thm}
	
		\begin{proof}
			Suppose that $P$ is an $I$-primary and $r \in R-\sqrt{(P:M)}$. Then by Theorem \ref{1} $(P:r)=P$ or $(P:r)=(IP:r)$. So by Lemma \ref{a2}, $(F \otimes P:r)=F \otimes (P:r)=F \otimes P$ or $(F \otimes P:r)=F \otimes (P:r)=F \otimes (IP:r)=(F \otimes IP:r)=(I(F \otimes P):r) $ and consequently $F \otimes P$ is an $I$-primary submodule of $F \otimes M$.
		\end{proof}

	\begin{prop} \label{111}
		Let $F$ be a faithfully flat $R$-module. Then a submodule $P$ of an $R$-module $M$ is $I$-primary if and only if $F \otimes P$ is an $I$-primary submodule of $F \otimes M$.
	\end{prop}

	\begin{proof}
		Suppose that $P$ is an $I$-primary submodule of an $R$-module $M$ and $F$ a faithfully flat $R$-module. If $F \otimes P=F \otimes M$, then the exactness of the sequence $ 0\longrightarrow F \otimes P \longrightarrow F \otimes M \longrightarrow 0$ imply the exactness of $0\longrightarrow P\longrightarrow M \longrightarrow 0$ and hence $P=M$ which is a contradiction. So $F \otimes P \neq F \otimes M$ and by Theorem \ref{110} $F \otimes P$ is an $I$-primary submodule of $ F \otimes M$. Conversely, let $F \otimes P$ be an $I$-primary submodule of $F \otimes M$. Hence we obtain $F \otimes P \neq F \otimes M$ and so $P \neq M$. By Lemma \ref{a2}, for every $r \in R-\sqrt{(P:M)}$ we have $r \notin (F \otimes P:F \otimes M)$ and $F \otimes (P:r)=(F \otimes P:r)=F \otimes P$ or $F \otimes (P:r)=(F \otimes P:r)=(I(F \otimes P):r)=(F \otimes IP:r)=F \otimes (IP:r)$. Suppose $F \otimes(P:r)=F \otimes P$. Then $0\longrightarrow F\otimes (P:r)\longrightarrow F \otimes P\longrightarrow 0$ is an exact sequence and as $F$ is faithfully flat, $0\longrightarrow (P:r)\longrightarrow P\longrightarrow 0$ is an exact sequence and consequently $(P:r)=P$. The other case can be proved similarly. Thus by Theorem \ref{1}, $P$ is an $I$-primary submodule of $M$.
	\end{proof}

		It is well-known that $I \otimes F \cong IF$ for any ideal $I$ of $R$ and flat $R$-module $F$. Thus by Theorem \ref{110} and  Proposition \ref{111} we conclude the following.

		\begin{cor}
			Let $F$ be a flat $R$-module and $P$ be an $I$-primary ideal of $R$ with $PF \neq F$. Then $PF$ is an $I$-primary submodule of $F$. In the case $F$ is faithfully flat, the converse is also true. 
		\end{cor}
		
		We know that every polynomial ring $R[x]$ is a flat $R$-module and that $R[x] \otimes M \cong M[x]$. Hence as an immediate consequence of the Theorem \ref{110} we give the following corollary.

		\begin{cor}
			Let $M$ be an $R$-module and $x$ an indeterminate. If $P$ is an $I$-primary submodule of $M$, then $P[x]$ is an $I$-primary submodule of $M[x]$.
		\end{cor}



\begin{thebibliography}{99}
              \bibitem{11-1} Akray, I., $I$-prime ideals, to be appear.
             \bibitem{22-2} Akray, I. and Hussein, H. S., $I$-prime submodules, to be appear.

            \bibitem{34} Ali, M., Multiplication modules and homogeneous idealization. II. Beitr Algebra Geom, 2007, 48: 321 - 343 
            
            \bibitem{123} Ameri, R., On the prime submodules of multiplication modules, Inter. J. Math. Math. Sci., 27 (2003), 1715 - 1724. 
            
            \bibitem{6-6} Anderson, D. D., Smith, E., Weakly prime ideals. Houston J. Math. 29 (2003), 831 - 840. 
            
       		
       	    \bibitem{114} Baziar M. and Behboodi M., Classical primary submodules and decomposition of modules. J Algebra Appl, 8(3) (2009), 351 - 362 
       	
       		\bibitem{124} Behboodi M., Classical prime submodules [PhD Thesis]. Ahvaz: Chamran University, Iran, 2004.
       		\bibitem{134} Behboodi M, Koohy H. Weakly prime modules. Vietnam J Math, 32(3) (2004), 303 - 317 
       		
       		\bibitem{5-5} Ebrahimi Atani, S., Farzalipour, F., On weakly primary ideals, Georgian Math. J. 12 (2005), 423 - 429.
       		\bibitem{1234} El-Bast, Z. A. and Smith, P. F., Multiplication modules, Comm. Algebra, 16 (1988), 755 - 779. 
       		
       	    
       	    
       	     \bibitem{164} Lu C. Prime submodules of modules. Comment Math Univ St Paul, 33(1) (1984), 61 - 69 
           
           
\end{thebibliography}

\end{document}